\documentclass[a4paper,11pt]{article}

\usepackage{amssymb,amsmath,amsthm,color}
\usepackage[T1]{fontenc}

\newcommand\de{\text{d}}
\newcommand\eq[1]{(\ref{eq:#1})}
\newtheorem{theorem}{Theorem}
\newtheorem{lemma}{Lemma}
\newtheorem{corollary}{Corollary}

\begin{document}

\title{Synchronized L\'evy Queues}
\author{Offer Kella\thanks{
Department of Statistics, The Hebrew University of Jerusalem, Mount Scopus, Jerusalem 91905, Israel (offer.kella@huji.ac.il);
supported in part by grant 1647/17 from the Israel Science Foundation and the Vigevani Chair in Statistics}
\ \ and\ \
Onno Boxma\thanks{
EURANDOM and Department of Mathematics and Computer Science, Eindhoven
University of Technology, P.O. Box 513, 5600 MB Eindhoven, The
Netherlands (o.j.boxma@tue.nl); research partly funded by an NWO TOP grant, Grant Number 613.001.352, and by
the NWO Gravitation project NETWORKS, Grant Number 024.002.003}}

\maketitle

\begin{abstract}
We consider a multivariate L\'evy process
where the first coordinate is a L\'evy process with no negative jumps which is not a subordinator and the others are nondecreasing.
We determine the Laplace-Stieltjes transform of the steady-state buffer content vector of an associated system of parallel queues.
The special structure of this transform allows us to rewrite it as a product of joint Laplace-Stieltjes transforms.
We are thus able to interpret the buffer content vector as a sum of independent random vectors.
\end{abstract}


\bigskip
\noindent {\bf AMS Subject Classification (MSC2010):} 60G51, 60K25, 90B05.

\section{Introduction}
\label{intro}
We consider $n$ parallel stations or queues which process some material that we call {\em fluid}.
The inputs to the stations are correlated. The net input vector is denoted by
$Y=\{(Y_1(t),\ldots,Y_n(t))|\,t\ge 0\}$, where $Y_i(t) = \sum_{j=1}^i X_j(t)$, $1 \le i \le n$, and where
$X_1$ is a L\'evy process with no negative jumps which is not a subordinator and $X_2,\dots,X_n$ are subordinators which are not identically zero.
Put differently: station $1$ receives a net input process $X_1$, and station $2$ receives {\em in addition} $X_2$, and station $3$ receives on top of that also $X_3$, etc.
Since the processes $X_2,\dots,X_n$ are subordinators, i.e., non-decreasing L\'evy processes, the net input to station $i$ is at least as large as the net input to station $i-1$,
$2 \le i \le n$.

This model of $n$ parallel stations generalizes the model studied in \cite{kella1993} in two respects. Firstly,
we do not require $X_1$ to be a subordinator minus a linear drift. Secondly, throughout the paper we allow $X_1,\dots,X_n$ to be dependent,
whereas in \cite{kella1993} independence was assumed in deriving some of its results.
Our model also generalizes the model of $n$ parallel queues studied in \cite{bbrw2014}, because the latter paper
restricts itself to compound Poisson inputs (and hence the net inputs are compound Poisson processes minus linear drifts). The steady state workload decomposition that we eventually identify is also related to results developed in \cite{ddr2007}.
For further literature on fluid networks we refer to the mini-survey \cite{BoxmaZwart} and references therein;
for linear stochastic fluid networks, see, e.g., \cite{kw1999}.
Background material on L\'evy processes with some emphasis on related applications can be found in, e.g.,
\cite{DebickiMandjes,Kyprianou}.

Our main results are the following.
We determine the Laplace-Stieltjes transform of the steady-state buffer content vector for the system of $n$ parallel queues.
The special structure of this buffer content transform allows us to rewrite it as a product of $n$ joint Laplace-Stieltjes transforms.
Each term of the product is given a probabilistic interpretation.
We are thus able to interpret the buffer content vector as a sum of $n$ independent random vectors.

The paper is organized as follows.
Section~\ref{model} contains a detailed model description and some preliminary results regarding the Laplace exponent of the multivariate L\'evy
process $X$. The main results are derived in Section~\ref{main}.
Two remarks at the end of that section point out that our main results
are also of immediate relevance for a tandem fluid model, a priority queue and a multivariate insurance risk model.

\section{The model and preliminaries}
\label{model}

Let $X=\{(X_1(t),\ldots,X_n(t))|\,t\ge 0\}$
be a multivariate L\'evy process with $X(0)=0$, where $X_1$ is a L\'evy process with no negative jumps which is not a subordinator and $X_i$ is a subordinator which is not identically zero for each $2\le i\le n$.
We have $E[{\rm e}^{-\alpha^T X(t)}] = {\rm e}^{\varphi(\alpha)t}$, where
the Laplace exponent $\varphi(\alpha)$ has the form
\begin{align}
\varphi(\alpha)=-a^T\alpha+\frac{\sigma^2}{2}\alpha_1^2+\int_{\mathbb{R}_+^n}\left(e^{-\alpha^T x}-1+\alpha^T x1_{(0,1]^n}(x)\right)\nu(\de x) ,
\end{align}
where if we denote $\nu_i(A)=\nu(\mathbb{R}_+^{i-1}\times A\times\mathbb{R}_+^{n-i})$ for $1 \leq i \leq n$, then $\int_{\mathbb{R}_+}(x^2_1\wedge 1)\nu_1(\de x_1)<\infty$ and $\int_{\mathbb{R}_+}(x_i\wedge 1)\nu_i(\de x_i)<\infty$ for $2\le i\le n$. For this setup we can rewrite the Laplace exponent in the following way:
\begin{align}
\varphi(\alpha)=-c^T\alpha+\frac{\sigma^2}{2}\alpha_1^2+\int_{\mathbb{R}_+^n}\left(e^{-\alpha^T x}-1+\alpha_1 x_1 1_{(0,1]}(x_1)\right)\nu(\de x) ,
\end{align}
where $c_1=a_1+\int_{(0,1]\times(\mathbb{R}_+^{n-1}\setminus(0,1]^{n-1})}x_1\nu(\de x_1)$ and $c_i=a_i-\int_{(0,1]^n}x_i\nu(\de x)$ for $2\le i\le n$.
Letting
\begin{align}
\varphi_1(\alpha_1)&=\varphi(\alpha_1,0,\ldots,0)\\
&=-c_1\alpha_1+\frac{\sigma^2}{2}\alpha_1^2+
\int_{\mathbb{R}_+}\left(e^{-\alpha_1x_1}-1+\alpha_1x_11_{(0,1]}(x_1)\right)\nu_1 (\de x_1)
\nonumber\end{align}
gives that
\begin{equation}
\varphi(\alpha)=\varphi_1(\alpha_1)-\sum_{i=2}^nc_i\alpha_i-\int_{\mathbb{R}_+^n}
e^{-\alpha_1x_1}\left(1-e^{-\sum_{i=2}^n\alpha_ix_i}\right)\nu(\de x)\ .
\label{eq4}
\end{equation}
It should be noted that the integral on the right is finite for every choice of $\alpha\in \mathbb{R}_+^n$. This follows from the fact that
\begin{equation}
e^{-\alpha_1x_1}\left(1-e^{-\sum_{i=2}^n\alpha_ix_i}\right)\le \left(\sum_{i=2}^n\alpha_ix_i\right)\wedge 1\,,
\end{equation}
so that
\begin{align}
&\int_{\mathbb{R}_+^n}
e^{-\alpha_1x_1}\left(1-e^{-\sum_{i=2}^n\alpha_ix_i}\right)\nu(\de x)\nonumber\\ &\qquad\le
\sum_{i=2}^n\alpha_i\int_{(0,1]^n}x_i\nu(\de x)+\nu\left(\mathbb{R}_+^n\setminus(0,1]^n\right)\\
&\qquad\le \sum_{i=2}^n \alpha_i \int_{(0,1]}x_i\nu_i(\de x)+\nu\left(\mathbb{R}_+^n\setminus(0,1]^n\right)<\infty\,.\nonumber
\end{align}
We note that we may also write
\begin{equation}
\varphi(\alpha)=\varphi_1(\alpha_1)+\varphi(0,\alpha_2,\ldots,\alpha_n)
+\int_{\mathbb{R}_+^n}
\left(1-e^{-\alpha_1x_1}\right)\left(1-e^{-\sum_{i=2}^n\alpha_ix_i}\right)\nu(\de x)\,,
\end{equation}
observing that
\begin{equation}\label{eq:eta}
\varphi(0,\alpha_2,\ldots,\alpha_n)=-\sum_{i=2}^nc_i\alpha_i-\int_{\mathbb{R}_+^n}
\left(1-e^{-\sum_{i=2}^n\alpha_i x_i}\right)\nu(\de x)\equiv-\eta(\alpha_2,\ldots,\alpha_n)
\end{equation}
is the Laplace exponent of $(X_2,\ldots,X_n)$, which implies (since it is an ($n-1$)-dimensional subordinator) that necessarily $c_i\ge 0$ for $2 \le i \le n$.

In a similar manner, letting, for $\beta\in\mathbb{R}_+$ and $2\le k\le n$,
\begin{equation}
\varphi_k(\beta)=\varphi(\underbrace{\beta,\ldots,\beta}_k,\underbrace{0,\ldots,0}_{n-k})\,,
\end{equation}
we have that $\varphi_k$ is the Laplace exponent of $\sum_{i=1}^kX_i$, which is (because of $X_1$) not a subordinator, and, from (\ref{eq4}),
\begin{align}
\varphi(\underbrace{\beta,\ldots,\beta}_k,\alpha_{k+1},\ldots,\alpha_n)\nonumber
&=\varphi_k(\beta)-\sum_{i=k+1}^nc_i\alpha_i\\
&\quad-\int_{\mathbb{R}_+^n}e^{-\beta\sum_{i=1}^kx_i}\left(1-e^{-\sum_{i=k+1}^n\alpha_ix_i}\right)\nu(\de x)\ .
\label{eq10}
\end{align}

\section{The main results}
\label{main}

In this section we determine the Laplace-Stieltjes transform of the steady-state buffer content vector for the system of $n$ parallel queues
(Theorem~\ref{thm1} and Corollary~\ref{cor1}).
The special structure of this buffer content transform allows us to rewrite it as a product of $n$ joint Laplace-Stieltjes transforms.
Each term of the product is given a probabilistic interpretation.
We are thus able to interpret the buffer content vector as a sum of $n$ independent random vectors.
We end the section with a few remarks concerning the relations between the model under consideration and
some other multivariate stochastic models.

\begin{lemma}
For $k=1,\ldots,n-1$, there is a unique positive $\beta=\psi_k(\alpha_{k+1},\ldots,\alpha_n)$
with $\alpha_{k+1},\dots,\alpha_n \geq 0$
such that
\begin{equation}\label{eq:psik}
\varphi(\beta,\ldots,\beta,\alpha_{k+1},\ldots,\alpha_n)=0\ .
\end{equation}
\label{Lemma1}
\end{lemma}
\begin{proof}
  Since $X_i$ is not identically zero for each $2\le i\le n$, then, when $\alpha_{k+1},\ldots,\alpha_n$ are not all zero,
it follows from (\ref{eq10}) and the fact that $c_2,\ldots,c_n \ge 0$, that
  \begin{equation}
  \varphi(0,\ldots,0,\alpha_{k+1},\ldots,\alpha_n)=-\sum_{i=k+1}^nc_i\alpha_i-\int_{\mathbb{R}_+^n}
  \left(1-e^{-\sum_{i=k+1}^n\alpha_ix_i}\right)\nu(\de x)<0\,.
\label{11}
   \end{equation}
Since $\sum_{i=1}^kX_i$ is not a subordinator, as $\beta \rightarrow \infty$ we have $\varphi_k(\beta)\to \infty$ and
 \begin{equation}
  \int_{\mathbb{R}_+^n}e^{-\beta\sum_{i=1}^kx_i}\left(1-e^{-\sum_{i=k+1}^n\alpha_ix_i}\right)\nu(\de x)\to 0 ,
  \end{equation}
  (dominated convergence) and thus, cf.\ (\ref{eq10}),
  $\varphi(\beta,\ldots,\beta,\alpha_{k+1},\ldots,\alpha_n)\to\infty$. Therefore, since $\varphi$ is convex (hence, $\varphi(\beta,\ldots,\beta,\alpha_{k+1},\ldots,\alpha_n)$ is convex in $\beta$),
and $\varphi(0,\ldots,0,\alpha_{k+1},\ldots,\alpha_n) <0$ according to (\ref{11}),
the statement of the lemma follows.
\end{proof}
\noindent
Now assume that $Y_i=\sum_{j=1}^i X_j$ for $1\le i\le n$. Then, for each $t\ge 0$, we have for each $s<t$ that
\begin{equation}
Y_1(t)-Y_1(s)\le Y_2(t)-Y_2(s)\le\ldots\le Y_n(t)-Y_n(s)\ .
\end{equation}
Clearly, $Y =\{(Y_1(t),\ldots,Y_n(t))|t \ge 0\}$ is also a L\'evy process with Laplace exponent
\begin{equation}
\tilde\varphi(\alpha)=\varphi(\alpha_1+\ldots+\alpha_n,\alpha_2+\ldots+\alpha_n,\ldots,\alpha_n)\ .
\end{equation}
Lemma~\ref{Lemma1} implies that for
\begin{equation}\label{eq:beta}
\beta=\psi_k(\alpha_{k+1}+\ldots\alpha_n,\ldots,\alpha_{n-1}+\alpha_n,\alpha_n)
-\sum_{i=k+1}^n\alpha_i\,,
\end{equation}
$\tilde\varphi(0,\ldots,0,\beta,\alpha_{k+1},\ldots,\alpha_n)$ is zero.

We are now ready to study the buffer content of a system of $n$ parallel fluid queues, with net input $Y$.
Let
\begin{align}
L_i(t)=-\inf_{0\le s\le t}Y_i(s), ~~~
Z_i(t)=Y_i(t)+L_i(t)\,.
\end{align}
Notice that $Z_i(t)$ can be viewed as the buffer content
of a queue with net L\'evy input $Y_i(t) = X_1(t)+\ldots + X_i(t)$, $i=1,\ldots,n$,
and $L_i(t)$ is the local time at level $0$ of that queue.

We necessarily have (see Theorem~6 of~\cite{kw1996}) that $Z_i(t)\le Z_{i+1}(t)$ for $1\le i\le n-1$.
Thus, if we assume that $EY_n(1)<0$ (hence $EY_i(1)<0$ for all $1\le i\le n$), then our system of $n$ parallel fluid queues has a stationary distribution.
We shall now determine the LST (Laplace-Stieltjes transform) of the steady state workload vector, to be denoted by $Z^*$.

\begin{theorem}
\label{thm1}
The LST of the steady state workload vector $Z^*$ is given by
\begin{equation}\label{eq:LST}
\tilde\varphi(\alpha)Ee^{-\alpha^TZ^*}=\sum_{k=1}^{n-1}\alpha_kf_k(\alpha_{k+1},\ldots,\alpha_n)+\alpha_nf_n\,,
\end{equation}
where $f_n=-EY_n(1)$ is constant and
$f_k(\alpha_{k+1},\ldots,\alpha_n)$
are recursively given by
\begin{equation}\label{eq:fk}
f_k(\alpha_{k+1},\ldots,\alpha_n)
=\frac{\sum_{i=k+1}^{n-1}\alpha_if_i(\alpha_{i+1},\ldots,\alpha_n)+\alpha_nf_n}
{\sum_{i=k+1}^n\alpha_i-\psi_k(\alpha_{k+1}+\ldots+\alpha_n,\ldots,\alpha_{n-1}+\alpha_n,\alpha_n)}\,,
\end{equation}
where an empty sum is defined to be zero.
\end{theorem}

\begin{proof}
If $Z^*$ has this distribution, then, since $Z_1^*\le\ldots\le Z_n^*$ (hence, $Z_i^*=0$ implies that $Z_j^*=0$ for $1\le j\le i$),  as in (2.12) of \cite{kella1993}, (\ref{eq:LST}) is satisfied for all $\alpha_1,\ldots,\alpha_n$ satisfying $\sum_{i=k}^n\alpha_i\ge 0$ for all $k=1,\ldots,n$.
Below we first determine $f_n$, and thereafter we show how $f_{n-1},\ldots,f_1$ can successively be determined;
$f_{n-1}$ is expressed in $f_n$, then $f_{n-2}$ is expressed in $f_{n-1}$ and $f_n$, etc.
Letting $\alpha_1=\ldots=\alpha_{n-1}=0$, we have, upon dividing by $\alpha_n$ and letting $\alpha_n\downarrow 0$, that
\begin{equation}
f_n=\frac{\partial \tilde\varphi}{\partial \alpha_n}(0)=-EY_n(1)=-\sum_{i=1}^n EX_i(1)\ .
\end{equation}
Now, note that if we let $\alpha_1=\ldots=\alpha_{k-1}=0$ then we have the formula
\begin{equation}\label{eq:Z}
\tilde\varphi(0,\ldots,0,\alpha_k,\ldots,\alpha_n)Ee^{-\sum_{i=k}^n\alpha_iZ^*_i}
=\sum_{i=k}^{n-1}\alpha_if_i(\alpha_{i+1},\ldots,\alpha_n)+\alpha_nf_n\ .
\end{equation}
Assume that $f_i(\alpha_{i+1},\ldots,\alpha_n)$ are known for $i=k+1,\ldots,n-1$. Set $\alpha_k$ to be the right hand side of \eq{beta}. Recall that $Z^*_{i-1}\le Z^*_i$ for $2\le i\le n$. Therefore
\begin{equation}
\sum_{i=k}^n\alpha_i Z_i^*=\psi_k(\alpha_{k+1}+\ldots,\alpha_n,\alpha_{n-1}+\alpha_n,\alpha_n)Z^*_k+\sum_{i=k+1}^n\alpha_i(Z_i^*-Z_k^*)  \geq 0,
\end{equation}
so that $Ee^{-\sum_{i=k}^n\alpha_iZ_i^*}\le 1$ for our choice of $\alpha_k$ (remember that we do not demand that all $\alpha_i \geq 0$).
This implies (\ref{eq:fk}).
In principle, we can thus compute the right side of \eq{LST} and hence $Ee^{-\alpha^TZ^*}$ for all $\alpha_1,\ldots,\alpha_n$ satisfying $\sum_{i=k}^n\alpha_i\ge 0$ for all $k=1,\ldots,n$.
It follows from (2.12) of \cite{kella1993}
that
$f_k(\alpha_{k+1},\ldots,\alpha_n)$
is a constant times some (joint) Laplace-Stieltjes transform for every $1\le k\le n-1$.
We shall also see this in Corollary~\ref{cor1}.
\end{proof}
\noindent
It should be mentioned that the above theorem generalizes Section~3 of \cite{kella1993} in two ways. The first is that it is not necessary to assume that $X_1$ is a subordinator minus a drift and the second is that it is not necessary to assume that $X_1,\ldots,X_n$ are independent.
It also generalizes results from \cite{bbrw2014} from the compound Poisson setting to the more general L\'evy subordinator setting.
The latter paper considers a system of $n$ queues which simultaneously receive input from an $n$-dimensional
compound Poisson process, and the jump sizes of the simultaneous jumps are stochastically ordered.
The steady state joint workload LST of that system is determined in \cite{bbrw2014}.
\\

\noindent
In \cite{bbrw2014}, furthermore, a decomposition is presented for the $n$-dimensional workload LST, and the terms of this decomposition are given an interpretation.
In the remainder of this section, our aims are to also decompose the $n$-dimensional LST of $Z^*$ into a product of terms,
and to give interpretations of these terms.
In particular, it would be nice to understand the meaning of $\varphi_k(\alpha_{k+1},\ldots,\alpha_n)$. For this it would suffice to consider $k=1$. We follow ideas from \cite{kella1993} where some unnecessary independence assumptions were made, but with a new observation which was missed at the time. The first thing to observe is the following.
Let
\begin{equation}
T_1(x)=\inf\{t|\,X_1(t)=-x\}.
\end{equation}
Then it is well known that $T_1(x)$ is a.s. finite for each $x\ge 0$ and in fact $\{T_1(x)|\,x\ge 0\}$ is a subordinator with Laplace exponent $-\varphi^{-1}_1(\cdot)$, where $\varphi^{-1}_1$ is the inverse of the strictly increasing and continuous function $\varphi_1$ (from $\mathbb{R}_+$ to $\mathbb{R}_+$).

Now, since
$e^{-\sum_{i=1}^n \alpha_i X_i(t)-\varphi(\alpha)t}$ is a mean one martingale (Wald martingale)
for each
$\alpha\in \mathbb{R}^n_+$,
this holds in particular for $\alpha_1=\psi_1(\alpha_2,\ldots,\alpha_n)$ in which case $\varphi(\alpha)=0$. Since $T_1(x)$ is a stopping time, the optional stopping theorem  implies that
\begin{equation}
E\exp\left(-\psi_1(\alpha_2,\ldots,\alpha_n)X_1(T_1(x)\wedge t)-\sum_{i=2}^n\alpha_iX_i(T_1(x)\wedge t)\right)=1\ .
\end{equation}
Noting that $\psi_1(\alpha_2,\ldots,\alpha_n)\ge 0$ and $X_1(T_1(x)\wedge t)\ge -x$, and that $X_i(T_1(x)\wedge t)\ge 0$ for each $2\le i\le n$, this implies by bounded convergence and the fact that $X_1(T_1(x))=-x$, that
\begin{equation}
Ee^{-\sum_{i=2}^n\alpha_iX_i(T_1(x))}=e^{-\psi_1(\alpha_2,\ldots,\alpha_n)x}\,.
\label{eq25}
\end{equation}
In \cite{kella1993} it was explained why $\{(X_2(T_1(x)),\ldots,X_n(T_1(x)))|,x\ge 0\}$ is a (nondecreasing) L\'evy process where it was stated that it is important that $X_2,\ldots,X_n$ are independent. This statement about independence was an oversight, as with a similar argument it holds that it is a L\'evy process without any independence assumptions. Formula (\ref{eq25}) implies that $-\psi_1(\alpha_2,\ldots,\alpha_n)$ is in fact the Laplace exponent of this L\'evy process. If $X_1$ is independent of $(X_2,\ldots,X_n)$ then, recalling \eq{eta},
\begin{equation}
E\exp\left(-\sum_{i=2}^n\alpha_iX_i(T_1(x))\right)=  E ~ e^{-\eta(\alpha_2,\ldots,\alpha_n)T_1(x)}
=e^{-\varphi_1^{-1}(\eta(\alpha_2,\ldots,\alpha_n))x} ,
\end{equation}
and thus for this case
\begin{equation}
\psi_1(\alpha_2,\ldots,\alpha_n)=\varphi^{-1}_1(\eta(\alpha_2,\ldots,\alpha_n))\,,
\end{equation}
which is what appears in \cite{kella1993} (with different notations). In fact, since $X_1(T_1(x))=-x$, we clearly have that $\{X(T_1(x))|\,x\ge 0\}$ is an $n$-dimensional L\'evy process with Laplace exponent $\alpha_1-\psi_1(\alpha_2,\ldots,\alpha_n)$. In particular, this implies that the Laplace exponent of $\{Y(T_1(x))|\,x\ge0\}$ is
\begin{equation}
\sum_{i=1}^n\alpha_i-\psi_1\left(\alpha_2 + \ldots + \alpha_n,\ldots,\alpha_{n-1}+\alpha_n,\alpha_n\right)\ .
\end{equation}
At this point it would be good to refer to \eq{beta}.

Now, we note that for $2\le i\le n$,
\begin{equation}
\inf_{0\le s\le T_1(x)}Y_i(s)=\inf_{0\le y\le x}Y_i(T_1(y))
\end{equation}
and thus
\begin{equation}
Z_i(T_1(x))=Y_i(T_1(x))-\inf_{0\le y\le x}Y_i(T_1(y))\,.
\end{equation}

In addition to the formal proof given in \cite{kella1993}, note that if $0\le s\le T_1(x)$ and
\begin{equation}
s\not\in \{T_1(y)|\,0\le y\le  x\}
\end{equation}
then, since $T_1(\cdot)$ is right continuous and $T_1(0)=0$, there must be some $y\in[0,x]$ for which $T_1(y)<s$ and $X_1(T_1(y))=-y<X_1(s)$. Since $X_2,\ldots,X_n$ are nondecreasing, $Y_i(T_1(y))<Y_i(s)$.
This means that the only contenders to minimize $Y_i$ on the interval $[0,T_1(x)]$ are $\{T_1(y)|\,0\le y\le x\}$.

Similar ideas imply that, with $T_k(x)=\inf\{t|\,Y_k(t)=-x\}$, \\ $\{(X_{k+1}(T_k(x)),\ldots,X_n(T_k(x)))|\,x\ge0\}$ is a subordinator with Laplace exponent $-\psi_k(\alpha_{k+1},\ldots,\alpha_n)$ and for every $k+1\le i\le n$ it follows that
\begin{equation}
Z_i(T_k(x))=Y_i(T_k(x))-\inf_{0\le y\le x}Y_i(T_k(y)),
\end{equation}
and we also observe that $Y_i(T_k(x))=\sum_{j=k+1}^iX_j(T_k(x))-x$ for $k+1 \le i \le n$.
That is, for each $k+1 \le i \le n$, $Y_i(T_k(x))$ is a subordinator minus a unit drift.

Letting $Z^{k*}$ denote a random vector having the limiting distribution of $Z(T_k(x))$, noting that necessarily $Z_1^{k*}=\ldots=Z_k^{k*}=0$, then from Corollary~2.3 of \cite{kella1993} we have that
\begin{equation}\label{eq:Zstar}
f_k(\alpha_{k+1},\ldots,\alpha_n)=-EY_k(1)E\exp\left(-\sum_{i=k+1}^n\alpha_iZ^{k*}_i\right)\ .
\end{equation}
The following corollary is now implied by Theorem~\ref{thm1}.
\begin{corollary}
\label{cor1}
\begin{equation}
\tilde\varphi(\alpha)Ee^{-\alpha^TZ^*}=\sum_{k=1}^{n}\alpha_k (-EY_k(1))
E\exp\left(-\sum_{i=k+1}^n\alpha_iZ^{k*}_i\right)\ ,
\end{equation}
with an empty sum being zero.
\end{corollary}
This confirms the statement, made in the proof of Theorem~\ref{thm1}, that all $f_k$ are some LST  of a (joint) distribution,
up to a multiplicative constant.

We emphasize that $Z^{k*}$ is a vector of workloads having the steady state distribution associated with the vector of workloads process embedded at time instants where station $k$ (hence also stations $1,\ldots,k-1$) is empty. For what follows, we recall that if $-\xi(\beta)$ is the Laplace exponent of some one-dimensional subordinator, then for some $b\in\mathbb{R}_+$ and L\'evy measure $\mu$ satisfying $\int_{\mathbb{R}_+}(u\wedge 1)\mu(\de u)<\infty$ we have that
\begin{equation}\label{eq:xi}
\xi(\beta)=b\beta+\int_{\mathbb{R}_+}\left(1-e^{-\beta u}\right)\mu(\de u) ,
\end{equation}
and when the mean $\xi'(0)$ is finite, then $\frac{\xi(\beta)}{\beta\xi'(0)}$ is the LST of the mixture, with weight factors
\begin{equation}\label{eq:mix}
\left(\frac{b}{b+\int_0^\infty\mu(u,\infty)\de u}\,,\,
\frac{\int_0^\infty\mu(u,\infty)\de u}{b+\int_0^\infty\mu(u,\infty)\de u}\right),
\end{equation}
of zero and a (residual lifetime) distribution having the density
\\
$\mu(x,\infty)/\int_0^\infty\mu(u,\infty)\de u$. See, e.g., (4.6) of \cite{kella1996}.

We first discuss the workload decomposition for the case $n=2$, exposing the key ideas;
thereafter, we briefly treat the case of a general $n$ via a repetition of the argument.
\\

\noindent
{\bf Workload decomposition for the $2$-dimensional case}
\\
For the case $n=2$ we first note that from $\varphi(\psi_1(\alpha_2),\alpha_2)=0$ it follows that with \begin{equation}
\eta_2(\alpha_2)=-\varphi(0,\alpha_2)=c_2\alpha_2+\int_{\mathbb{R}_+}\left(1-e^{-\alpha_2 x}\right)\nu_2(\de x) ,
\end{equation}
we have that
\begin{equation}
\psi_1'(0)=\frac{\eta_2'(0)}{\varphi_1'(0)} .
\end{equation}
It may be seen after some very simple manipulations, that for
$\alpha_2\ge 0$ and $\alpha_1\ge -\alpha_2$, with
$\alpha_1\not=\psi(\alpha_2)$:
\begin{equation}\label{Z1Z2}
Ee^{-\alpha_1Z_1^*-\alpha_2Z_2^*}=\frac{\varphi_1'(0)(\alpha_1+\alpha_2-\psi_1(\alpha_2))}
{\varphi(\alpha_1+\alpha_2,\alpha_2)}
\cdot\frac{(1-\psi_1'(0))\alpha_2}{\alpha_2-\psi_1(\alpha_2)}\ .
\end{equation}
An identical formula is given, employing different methods, in Proposition~1 of \cite{bbrw2014}.
That paper restricts itself to the special case where, with $J$ a two-dimensional compound Poisson process
with nonnegative (two-dimensional) jumps, we have $X(t) = J(t)-(t,0)$ (so that $Y(t) = (J_1(t)-t,J_1(t)+J_2(t)-t)$).

If we denote $W_1^*=Z_1^*$ and $W_2^*=Z_2^*-Z_1^*$ (nonnegative) then
\begin{equation}\label{eq:W1W2}
Ee^{-\alpha_1W_1^*-\alpha_2W_2^*}=\frac{\varphi_1'(0)(\alpha_1-\psi_1(\alpha_2))}
{\varphi(\alpha_1,\alpha_2)}
\cdot\frac{(1-\psi_1'(0))\alpha_2}{\alpha_2-\psi_1(\alpha_2)}\ .
\end{equation}
The second expression in the product on the right hand side is the LST of $W_2^{1*} ~= ~ Z_2^{1*}$ (which has the steady state distribution of the workload in station $2$ observed only at times when station $1$ is empty)
and from its form (generalized Pollaczek-Khinchin formula, {\em e.g.,} (2.5) of \cite{kella1996} among many others), it is indeed the steady state LST of a reflected process of a L\'evy fluid queue with subordinator input having the Laplace exponent $-\psi_1(\alpha_2)$ and a processing rate of one.
That is, it is the steady state LST of the reflected process associated with the driving process $\{\hat{J}(x)-x|\,x\ge 0\}$ where $\hat{J}(x)=X_2(T_1(x))$.
In the compound Poisson setting of \cite{bbrw2014}, it was observed to be the steady state LST of the workload in an $M/G/1$ queue with service times
distributed as the extra (compared with queue $1$) workload accumulated in station $2$ during a busy period of station $1$.

Setting $\alpha_1=0$, we can rewrite the resulting equation as follows
\begin{equation}
\frac{\eta_2(\alpha_2)}{\alpha_2\eta_2'(0)}Ee^{-\alpha_2W^*_2}=\frac{\psi_1(\alpha_2)}{\alpha_2\psi_1'(0)}
Ee^{-\alpha_2 W_2^{1*}}\ .
\label{eq41}
\end{equation}
That is, if $\xi_2^*$ and $\xi_2^{1*}$ have LST's $\frac{\eta_2(\alpha)}{\alpha_2\eta_2'(0)}$ and $\frac{\psi_1(\alpha_2)}{\alpha_2\psi_1'(0)}$ (see the paragraph that includes \eq{xi} and \eq{mix}) and are respectively independent of $W_2^*$ and $W_2^{1*}$, then we have the following decomposition:
\begin{equation}
\xi^*_2+W^*_2\sim \xi_2^{1*}+W_2^{1*}\ .
\label{eq42}
\end{equation}
We note that this is different from the decomposition described in Theorem~4.2 of \cite{kella1993} which holds when $X_1$ and $X_2$ are independent processes, but not otherwise. It is easy to check that with this independence, this decomposition holds here as well, where, unlike there, $X_1$ need not be a subordinator minus some drift.

Our next goal is to identify a joint distribution with LST given by the first expression of the product on the right hand side of \eq{W1W2}. First we observe from Corollary~(2.1) of \cite{kella1993} and the facts that $Z_1(T_1(x))=0$, $Z_1(t)=0$ for points of (right) increase of $L_1$ and $Z_1(t)=Z_2(t)=0$ for points of (right) increase of $L_2$ ({\em e.g.,} \cite{kella2006}), that
\begin{align}\label{eq:Z1Z22}
\tilde \varphi(\alpha)E\int_0^{T_1(x)}e^{-\alpha^TZ(s)}{\rm d}s&=Ee^{-\alpha_2 Z_2(T_1(x))}-1+\alpha_2EL_2(T_1(x))
\nonumber \\
&\quad +\alpha_1 E\int_0^{T_1(x)}e^{-\alpha_2 Z_2(s)}{\rm d}L_1(s) .
\end{align}
And in particular this holds when we substitute $\alpha_1 = \psi_1(\alpha_2) - \alpha_2$ (cf.\ (\ref{eq:beta})) in which case $\tilde\varphi(\alpha)=0$.
Therefore, subtracting \eq{Z1Z22} with $\alpha_1=\psi_1(\alpha_2)-\alpha_2$ from \eq{Z1Z22} and noting (see also Corollary~2.1 of \cite{kella1993}) that $\int_0^{T_1(x)}e^{-\alpha_2 Z_2(s)}{\rm d}L_1(s)=\int_0^xe^{-\alpha_2 Z_2(T_1(y))}{\rm d}y$, we have that
\begin{equation}
\tilde \varphi(\alpha)E\int_0^{T_1(x)}e^{-\alpha^TZ(s)}{\rm d}s=(\alpha_1+\alpha_2-\psi_1(\alpha_2))\int_0^xEe^{-\alpha_2 Z_2(T_1(y))}{\rm d}y\,.
\end{equation}
Dividing by $ET_1(x)=x/\varphi_1'(0)$ and recalling that $\tilde\varphi(\alpha)=\varphi(\alpha_1+\alpha_2,\alpha_2)$ now gives
\begin{equation}
\frac{1}{ET_1(x)}E\int_0^{T_1(x)}e^{-\alpha^TZ(s)}{\rm d}s=
\frac{\varphi_1'(0)(\alpha_1+\alpha_2-\psi_1(\alpha_2))}{\varphi(\alpha_1+\alpha_2,\alpha_2)}
\frac{1}{x}\int_0^xEe^{-\alpha_2 Z_2(T_1(y))}{\rm d}y\,.
\label{eq45}
\end{equation}
We now observe two facts. The first is that by regenerative theory the left hand side is the LST of the steady state distribution of a corresponding regenerative process where at time $T_1(x)$ the process is restarted. At this time $Z_1(T_1(x))=0$ and the remaining quantity at the second station is lost. The second fact is that as $x\downarrow 0$ we have by bounded convergence and the fact that $Z(T_1(\cdot))$ is right continuous with $Z(T_1(0))=0$ (since $T_1(0)=0$), that
\begin{equation}
\lim_{x\downarrow 0}\frac{1}{x}\int_0^x Ee^{-\alpha_2Z_2(T_1(y))}{\rm d}y=\lim_{x\downarrow 0}Ee^{-\alpha_2 Z_2(T_1(x))}=1\ .
\end{equation}
Using (\ref{eq45}), this implies that
\begin{equation}
\lim_{x\downarrow 0}\frac{1}{ET_1(x)}E\int_0^{T_1(x)}e^{-\alpha^TZ(s)}{\rm d}s=\frac{\varphi_1'(0)(\alpha_1+\alpha_2-\psi_1(\alpha_2))}{\varphi(\alpha_1+\alpha_2,\alpha_2)} ,
\end{equation}
where the right hand side converges to~$1$ as $(\alpha_1,\alpha_2)\to (0,0)$.
By the continuity theorem for LST's we have that necessarily the right hand side is an LST of a nonnegative random vector (for any $\alpha_2\ge 0$ and $\alpha_1\ge -\alpha_2$).
This also implies that for $\alpha_1,\alpha_2\ge 0$,
with $W(s) =(W_1(s),W_2(s)) := (Z_1(s),Z_2(s)-Z_1(s))$,
\begin{equation}\label{eq:decompW}
\lim_{x\downarrow 0}\frac{1}{ET_1(x)}E\int_0^{T_1(x)}e^{-\alpha^TW(s)}{\rm d}s=\frac{\varphi_1'(0)(\alpha_1-\psi_1(\alpha_2))}{\varphi(\alpha_1,\alpha_2)} ,
\end{equation}
so that this is also an LST. For the compound Poisson case, this implies that this is the LST of a version of this process such that, whenever $Z_1(t)=0$, any quantity available in queue~2 is lost. For this case, this interpretation was discovered in \cite{bbrw2014} and it continues to be valid if $X_1$ is a compound Poisson process with a negative drift and $X_2$ is a general subordinator ($X_1,X_2$ can be dependent).
\\

\noindent
{\bf Remark 1.}
To obtain moments of $W_1^*,W_2^*$ or $Z_1^*=W_1^*$, $Z_2^*=W_1^* + W_2^*$ requires slightly tedious but straightforward calculations.
Firstly, cf.\ (\ref{eq:W1W2}),
\begin{equation}
E({\rm e}^{-\alpha_1 W_1^*}) = \frac{\phi_1'(0)\alpha_1}{\phi(\alpha_1,0)}
\end{equation}
yields
\begin{equation}
EW_1^* = \frac{\phi_1''(0)}{2 \phi_1'(0)} = \frac{{\rm Var} X_1(1)}{-2 EX_1(1)} .
\end{equation}
Secondly, again from (\ref{eq:W1W2}),
\begin{equation}
E({\rm e}^{-\alpha_2 W_2^*}) =
- \frac{\varphi_1'(0)\psi_1(\alpha_2)}
{\varphi(0,\alpha_2)}
\cdot\frac{(1-\psi_1'(0))\alpha_2}{\alpha_2-\psi_1(\alpha_2)}\
\end{equation}
yields
\begin{equation}
EW_2^* =
\frac{\frac{d^2}{dy^2} \phi(0,y)}{2 \frac{d}{dy} \phi(0,y)}|_{y=0}
-\frac{\psi_1''(0)}{2 \psi_1'(0)}
- \frac{\psi_1''(0)}{2(1-\psi_1'(0))} ,
\end{equation}
which term by term corresponds  to (cf.\ (\ref{eq41}) and (\ref{eq42}))
\begin{equation}
EW_2^* = - E \xi_2^{*} + E\xi_2^{1*} + EW_2^{1*} .
\end{equation}
Insertion of
\begin{equation}
\frac{{\rm d}}{{\rm dy}} \phi(0,y)|_{y=0} = -EX_2(1), ~~~ \frac{{\rm d^2}}{{\rm dy^2}} \phi(0,y) = {\rm Var} X_2(1),
\end{equation}
\begin{equation}
\psi_1'(0) = - \frac{EX_2(1)}{EX_1(1)},
\end{equation}
\begin{eqnarray}
\psi_1''(0) &=&
-\frac{\phi_1''(0)}{\phi_1'(0)} (\psi_1'(0))^2 -\frac{\frac{{\rm d^2}}{{\rm dy^2}} \phi(0,y)|_{y=0}}{\phi_1'(0)} - 2\frac{{\rm d^2}}{{\rm dxdy}} \phi(x,y)|_{x=y=0} \frac{\psi_1'(0)}{\phi_1'(0)}
\nonumber
\\
&=&
\frac{{\rm Var} X_1(1)}{EX_1(1)} (\frac{EX_2(1)}{EX_1(1)})^2 + \frac{{\rm Var} X_2(1)}{EX_1(1)}
\nonumber
\\
&-& 2 {\rm cov}(X_1(1),X_2(1)) \frac{EX_2(1)}{(EX_1(1))^2} ,
\end{eqnarray}
allows one to express the moments of $W_1^*$ and $W_2^*$ into the first two moments and covariance of $X_1(1)$ and $X_2(1)$.
Compared to Formula (4.6) of \cite{kella1993}, the expression for $EW_2^*$ contains an extra term that includes ${\rm cov}(X_1(1),X_2(1))$.
\\

\noindent
{\bf Workload decomposition for the $n$-dimensional case}
\\
The above argument may be repeated for the $n$-dimensional case. Indeed, it follows from \eq{Z}, \eq{fk} and \eq{Zstar} that for $\alpha$ such that $\sum_{i=k}^n\alpha_i \ge 0$ for all $k=1,\ldots,n$,
\begin{align}
Ee^{-\alpha^TZ^*}&=\frac{\varphi_1'(0)
\left(\sum_{i=1}^n\alpha_i-\psi_1(\alpha_2+\ldots+\alpha_n,\ldots,\alpha_{n-1}+\alpha_n,\alpha_n)\right)}
{\tilde\varphi(\alpha)}\nonumber\\
&\qquad\cdot E\exp\left(-\sum_{i=2}^n\alpha_iZ_i^{1*}\right)\ ,
\end{align}
or equivalently, introducing $\alpha_i := \alpha_i - \alpha_{i+1}$ for $1 \leq i \leq n-1$, we have  for $\alpha_i\ge 0$ for $i=1,\ldots,n$,
\begin{align}
Ee^{-\alpha^TW^*}&=\frac{\varphi_1'(0)
\left(\alpha_1-\psi_1(\alpha_2,\ldots,\alpha_{n-1},\alpha_n)\right)}
{\varphi(\alpha)}\nonumber\\
&\qquad\cdot E\exp\left(-\sum_{i=2}^n\alpha_i W_i^{1*}\right)\ .
\end{align}
By induction, the right hand side may be written as a product of $n$ (joint) LST's where for $1\le k\le n-1$ the $k$th multiple is
\begin{equation}
\frac{\left(1-\frac{\partial\psi_{k-1}}{\partial\alpha_k}(0)\right)(\alpha_k-\psi_k(\alpha_{k+1},\ldots,\alpha_n))}
{\alpha_k-\psi_{k-1}(\alpha_k,\ldots,\alpha_n)} ,
\end{equation}
and the last one is
\begin{equation}
\frac{(1-\psi_{n-1}'(0))\alpha_n}{\alpha_n-\psi_{n-1}(\alpha_n)}\ .
\end{equation}
From \eq{psik} it follows that for $2\le k\le n$,
\begin{equation}
1-\frac{\partial\psi_{k-1}}{\partial\alpha_k}(0)=\frac{\varphi_1'(0)-\sum_{i=2}^k\eta_i'(0)}
{\varphi_1'(0)-\sum_{i=2}^{k-1}\eta_i'(0)}=\frac{-EY_k(1)}{-EY_{k-1}(1)}\ ,
\end{equation}
where for $k=2$, the empty sum in the denominator is defined to be zero.

For the case where $X_1$ is a subordinator minus a drift and $X_1,\ldots,X_n$ are independent, it seems that this decomposition can be inferred from more general results reported in Theorem~6.1 of \cite{ddr2007}.

Finally, the ideas that led to \eq{decompW} may be repeated to conclude that each multiplicative factor participating in this decomposition is indeed a (joint) LST.
\\

\noindent
{\bf Remark 2}.
The results of this section are also of immediate relevance for a tandem fluid model,
viz., a model of $n$ stations in series, in which material or work leaves each station as a fluid.
Such a connection between parallel stations and stations in series was already pointed out in \cite{kella1993}.
Tandem fluid models were introduced in \cite{kw1992book}.
The following $n$-station tandem fluid model is introduced and studied in that paper.
The input process of the first station is a nondecreasing L\'evy process and the $j$-th station receives the output of the ($j-1$)-st station at a constant rate $r_{j-1}$, as long as that station is not empty.
It was assumed that $r_1 \ge \dots \ge r_n$ to avoid the trivial case that a station is always empty.
That tandem fluid model was generalized in \cite{kella1993} by allowing additional external inputs to stations $2,\dots,n$.
Those inputs were assumed to be subordinators and, for most of the results, they were assumed to be independent from each other and from the input to station $1$.
It was observed in \cite{kella1993} that there is a direct connection between the workloads in this tandem fluid model and the workloads in a model of $n$ parallel stations.
The same observation, but for the case of compound Poisson inputs  (and allowing dependence between the input processes)
was also made in \cite{bbrw2014}.
That connection also extends to the case of dependent external L\'evy inputs.
More precisely, let $X_1,\dots,X_n$ be the external inputs to stations $1,\dots,n$ of the tandem fluid model, with $X_2,\dots,X_n$ being subordinators, and let $W_1,\dots,W_n$ denote their buffer content level processes.
Then we can identify $W_j(t)$ with $ Z_j(t)-Z_{j-1}(t)$, $1 \le j \le n$ (hence $W_1(t)+\dots+W_j(t) = Z_j(t)$), where $Z_j(t)$ as before denotes the buffer content level of station $j$ in the system of parallel stations studied in the present paper, and $Z_0(t) \equiv 0$.
\\

\noindent
{\bf Remark 3}.
In \cite{kella1993} an equivalence between the tandem fluid model and a particular single server priority queue is also pointed out.
Assume a compound Poisson input vector $X$ of customer classes $1,\ldots,n$.
Class $i$ has preemptive resume priority over class $j$ if $i<j$;
the total workload in the first $j$ classes can now be identified with $Z_j^*$.

In \cite{bbrw2014} a multivariate duality relation is established between (i) the model of $n$ parallel $M/G/1$ queues with simultaneous arrivals,
with stochastic ordering of the $n$ service times of each arriving batch
and (ii) a Cram\'er-Lundberg insurance risk model featuring $n$ insurance companies with simultaneous claim arrivals and stochastic ordering of those claims.
In particular, when the arrival processes in both systems have the same distribution, the joint steady-state workload distribution
${\mathbb P}(V_1 \le x_1,\dots,V_n \le x_n)$ in the queueing model equals the survival probability for all companies, with initial capital vector $(x_1,\dots,x_n)$.
This is a generalization of a well-known duality that is discussed, e.g., on  p.\ 46 of \cite{AsmAlb}.
Using the sample path argument presented there, one should also be able to generalize this duality to the case of the L\'evy input process of the present paper,
thus also obtaining the survival probability for $n$ insurance companies with the same $n$-dimensional input process as the $n$ parallel stations
of the present paper.
\\

\noindent
{\bf Acknowledgment}.
The authors thank Liron Ravner for interesting discussions.

\end{document}